\documentclass[11pt,reqno]{amsart}
\usepackage[margin= 2.7cm]{geometry}
\usepackage{algorithm}
\usepackage{algpseudocode}

\usepackage{amsmath}
\usepackage{amsfonts}
\usepackage{amssymb}
\usepackage{amsthm}
\usepackage{mathtools} % An improvement of amsmath
\usepackage{latexsym}
\usepackage{enumerate}
\usepackage{cancel}
\usepackage{ragged2e}
\usepackage{blindtext}
\usepackage{cases}
\usepackage{empheq}
\usepackage{multicol}
\usepackage{placeins} % Allows float barriers for spacing

\usepackage{wrapfig}
\usepackage{caption}

\usepackage{subfigure}
\usepackage{float}

\usepackage{color}

\usepackage[color=white,linecolor=black]{todonotes}
\usepackage{mathrsfs}
\usepackage{fontenc} %T1 font encoding
\usepackage{inputenc} %UTF-8 support
\usepackage{enumitem}
\usepackage{verbatim}

\theoremstyle{plain}
\newtheorem{theorem}{Theorem}[section]

\newtheorem*{notation}{Notation}
\theoremstyle{definition}

\theoremstyle{remark}

\usepackage[square,comma,numbers,sort]{natbib}
\usepackage[colorlinks=true, pdfborder={ 0 0 0}]{hyperref}
\hypersetup{urlcolor=blue, citecolor=red}
\usepackage{url}
\makeatletter
\renewcommand{\section}{\@startsection{section}{1}{0pt}%
  {1.5ex plus .2ex minus .2ex}%
  {1.0ex plus .2ex}%
  {\normalfont\large\bfseries\raggedright}}
\renewcommand{\subsection}{\@startsection{subsection}{2}{0pt}%
  {1.25ex plus .2ex minus .2ex}%
  {1.0ex plus .2ex}%
  {\normalfont\normalsize\bfseries\raggedright}}
\makeatother

\allowdisplaybreaks
\makeatletter
\def\abstractname{Abstract}
\renewenvironment{abstract}{%
  \begin{center}
    \textbf{\abstractname}
  \end{center}
}{}
\makeatother

\begin{document}
\title[NSE-Plate Interaction PDE system]{Analysis and Numerical Approximation to Interactive Dynamics of Navier Stokes-Plate Interaction PDE System}
% ======================  Author Information ======================
\author[P. G. Geredeli]{Pelin G. Geredeli}
\address[Pelin G. Geredeli]{School of Mathematical and Statistical Sciences, 
                Clemson University, 
       Clemson, SC, 29634, USA}
\email[Pelin G. Geredeli]{pgerede@clemson.edu
}
\author[Q. Lin]{Quyuan Lin}
\address[Quyuan Lin]{School of Mathematical and Statistical Sciences, 
                Clemson University, 
       Clemson, SC, 29634, USA}
\email[Quyuan Lin]{quyuanl@clemson.edu}
\author[D. McKnight]{Dylan McKnight}
\address[Dylan McKnight]{Department of Mathematics and Statistics, 
                Colorado Mesa University,Grand Junction, CO 81501, USA}
\email[Dylan McKnight]{dmcknight@coloradomesa.edu}
\author[M. M. Rahman]{Mohammad Mahabubur Rahman}
\address[Mohammad Mahabubur Rahman]{School of Mathematical and Statistical Sciences, 
                Clemson University, 
       Clemson, SC, 29634, USA}
\email[Mohammad Mahabubur Rahman]{rahman6@clemson.edu}

\maketitle{} 

\begin{abstract}

We consider a Navier-Stokes fluid-plate interaction (FSI) system which describes the evolutions of the fluid contained within a 3D cavity, as it interacts with a deformable elastic membrane on the ``free" upper boundary of the cavity. These models arise in various aeroelastic and biomedical applications as well as in the control of ocular pressure, and sloshing phenomena. We analyze the well-posedness of weak solutions to the stationary ($\lambda$-parametrized) coupled PDE system by way of invoking the nonlinear generalization of the abstract variational formulations which was introduced in \cite{girault2012finite}, wherein an inf-sup approach is followed to show existence-uniqueness of solutions under a small data assumption. 

In addition, we provide a numerical approximation scheme of the infinite dimensional coupled system via a finite element method approximation (FEM). The numerical results use a standard conforming scheme and handle the introduced nonlinearities via Picard iterations. Numerical results are obtained for an appropriate test problem satisfying the necessary boundary conditions and coupling. Moreover, error bounds between the FEM and theoretical solution in terms of the characteristic mesh size are supplied in appropriate Sobolev norms which agree with the established literature. These FEM approximations of the coupled system with their associated error bounds validate the theoretical findings.

\vskip.3cm \noindent \textbf{Key terms:} Fluid-Structure Interaction, Weak Solution, Picard Iteration, FEM Approximation  

\vskip.3cm \noindent \textbf{Mathematics Subject Classification (MSC): 35B35, 35B40, 35G35}

\end{abstract}

\section{Introduction}
\vspace{0.1cm}

\noindent Over the past decade, there has been growing interest in the qualitative and quantitative study of partial differential equation (PDE) systems that describe the interaction of a fluid flow with a deformable elastic structure. These coupled systems arise in various aeroelastic and biomedical applications; e.g., the fluttering of an airplane/airfoil wing, the movement of wind
turbines and bridges, blood transportation processes within arterial walls, the control of ocular pressure, and the phenomenon of sloshing \cite{bermudez2003finite, barbu2007existence, bukavc2016nonlinear, amara2002bending,avalos2008new,avalos2007coupled,barbu2008smoothness,bodnar2014fluid,kukavica2010strong,chambolle2005existence,du2003analysis,hsu2011blood}. In particular, the fluid-structure interaction (FSI) system \eqref{1}-\eqref{40} below describes the vibrations of an elastic membrane on the ``free" upper boundary of a 3D cavity which is filled with Navier-Stokes fluid. We describe here our PDE model. \\

\noindent \textbf{Navier-Stokes Fluid-Plate Interaction PDE Model}
\vspace{0.3cm}

\noindent We consider a coupled PDE system which constitutes an interaction between a coupled Navier-Stokes fluid and Kirchoff plate equation on a given geometry (see Figure 1 below). Here the fluid domain $\Omega_f \subset \mathbb{R}^3$  will be a Lipschitz domain with boundary $\partial \Omega_f := \bar{\Omega}_p \cup \bar{S},$ and $\Omega_p \cap S=\emptyset$. Moreover, the active and flat part $\Omega_p \subset \mathbb{R}^2$ represents the equilibrium position of the elastic domain, where the interaction between the respective fluid and structure components takes place.  In particular, $\Omega_p \subset \{x\in \mathbb{R}^3: x=(x_1,x_2,0)\}.$ Also, the surface $S \subset \{x\in \mathbb{R}^3:  x = (x_1, x_2, x_3) : x_3 \leq 0 \}$ is considered to be the inactive portion of the domain, where no interaction takes place. 
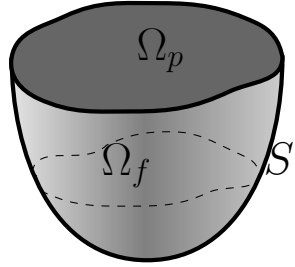
\begin{figure}[H]
%\begin{subfigure}[H]{0.3\linewidth}
\centering
\begin{tikzpicture}[scale=0.9]
\draw[left color=black!10,right color=black!20,middle
color=black!50, ultra thick] (-2,0,0) to [out=0, in=180] (2,0,0)
to [out=270, in = 0] (0,-3,0) to [out=180, in =270] (-2,0,0);

\draw [fill=black!60, ultra thick] (-2,0,0) to [out=80,
in=205](-1.214,.607,0) to [out=25, in=180](0.2,.8,0) to [out=0,
in=155] (1.614,.507,0) to [out=335, in=100](2,0,0) to [out=270,
in=25] (1.214,-.507,0) to [out=205, in=0](-0.2,-.8,0) [out=180,
in=335] to (-1.614,-.607,0) to [out=155, in=260] (-2,0,0);

\draw [dashed, thin] (-1.7,-1.7,0) to [out=80, in=225](-.6,-1.3,0)
to [out=25, in=180](0.35,-1.1,0) to [out=0, in=155] (1.3,-1.4,0)
to [out=335, in=100](1.65,-1.7,0) to [out=270, in=25] (0.9,-2.0,0)
to [out=205, in=0](-0.2,-2.2,0) [out=180, in=335] to (-1.514,-2.0)
to [out=155, in=290] (-1.65,-1.7,0);

\node at (0.2,0.1,0) {{\LARGE$\Omega_p$}};

\node at (1.95,-1.5,0) {{\LARGE $S$}};

\node at (-0.3,-1.6,0) {{\LARGE $\Omega_f$}};
\end{tikzpicture}

%\caption{Edge Domain}
\caption{Fluid-Structure Interaction Geometry }
\end{figure}
\noindent The PDE system under consideration is:
\begin{equation}
\label{1}
\begin{cases}
u_t - \nu \Delta u +(u \cdot \nabla)u+ \nabla p = f & \text{in} \ \Omega_f \times (0, T), \\
\nabla \cdot u = 0 & \text{in} \ \Omega_f \times (0, T), \\
u = 0 & \text{on} \ S.
\end{cases}
\end{equation}
\begin{align}
\label{2}
\begin{cases}
w_{tt} - \rho \Delta w_{tt} + \Delta^2 w =\frac{1}{2}u_3^2 +p \rvert_{\Omega_p} & \text{in} \ \Omega_p \times (0, T), \\
 w = \frac{\partial w}{\partial n_p} = 0 & \text{on} \ \partial \Omega_p,
\end{cases}
\end{align}
\begin{align}
\label{3}
 u = [u_1, u_2, u_3] = [0, 0, w_t] \quad \text{on} \ \Omega_p \times (0, T),
 \end{align}
 \begin{align}
 \label{40}
[u(0),w(0),w_t(0)]=[u_0,w_0,w_{1_0}].
\end{align}
Here, variables $ u(x, t),$ $p(x, t),$ and $w(x, t)$ represent the fluid velocity, fluid pressure, and plate displacement, respectively. Also, $\nu\geq 0$ is the constant viscosity of the fluid and $f(x, t)$ denote the body force acting on it. For simplicity, we take $f = 0$ throughout the analysis. The constant $\rho \geq 0$ is a rotational inertia parameter which is proportional to the square of the plate's thickness. When  $\rho = 0$, it signifies the absence of rotational dynamics in the plate (i.e., the thickness of the plate can be ignored). In addition, $n_p$ is the unit normal vector exterior to $\partial \Omega_p$.\\

\medskip

\noindent\textbf{Plan of the Paper and Notation}\\ \\
  In Section 1, we introduce the Navier Stokes fluid-plate interaction PDE system and describe the geometry of the model. In Section 2, we review the literature with relevant previous works. In Section 3, we discuss the novelties of our current manuscript and key challenges we encounter during the theoretical and numerical analysis of the associated PDE system. In Section 4, we describe the abstract spaces and present the preliminary concepts pertinent to our study. Moreover, we define the static $\lambda$-parametrized PDE system which we invoke for our numerical computation. In Section 5, we state our main results concerning the existence and uniqueness of solutions to the static Navier-Stokes fluid-plate interaction PDE system, and provide the detailed proofs. Section 6 is devoted to the numerical implementation, where we demonstrate the stability of the model. Finally, in Appendix, we present an abstract framework that we utilize to establish the existence and uniqueness of weak solutions to the static PDE system. 
\begin{notation}
For the remainder of the text, norms $||\cdot ||_{D}$ are taken to be $%
L^{2}(D)$ for the domain $D$. Inner products in $L^{2}(D)$ are written as $%
(\cdot ,\cdot )_{D}$, and the inner products $L^{2}(\partial D)$ are written as $%
( \cdot ,\cdot )_{\partial D} $. The space $H^{s}(D)$ will denote the Sobolev
space of order $s$ and $H_{0}^{s}(D)$ denotes the
closure of $C_{0}^{\infty }(D)$ in the $H^{s}(D)$ norm, which we denote by $%
\Vert \cdot \Vert _{H^{s}(D)}$ or $\Vert \cdot \Vert _{s,D}$. 
For the sake of simplicity, we denote $$\mathbf{L^2}(\Omega_f)=[L^{2}(\Omega _{f})]^3,~~\mathbf{H^1}(\Omega_f)=[H^{1}(\Omega _{f})]^3,~~\mathbf{H}_0^2(\Omega_p)=[H_0^{2}(\Omega _{p})]^2.$$ 
Also, we make use of the standard notation for the trace of functions defined on a Lipschitz 
domain $D$, i.e., for a scalar function $\phi \in H^{1}(D)$, we denote $%
\gamma (w)$ to be the trace mapping from $H^{1}(D)$ to $H^{1/2}(\partial D)$%
. %That is, $\gamma (w)=w\rvert_{\partial D},$ for $w\in C^{\infty}(D).$ 
We denote pertinent duality pairings as $(\cdot ,\cdot
)_{X\times X^{\prime }}$, the space $\widehat{\mathbf{L}^2}(\Omega)$ as 
\begin{align*}
\widehat{\mathbf{L}^2}(\Omega) := \{ u \in \mathbf{L}^2(\Omega) \mid \int_{\Omega} u(x') \, dx' = 0 \},
\end{align*}  
and also the Sobolev space $\widehat{\mathbf{H}^s_0}(\Omega) := \mathbf{H}_0^s(\Omega) \cap \widehat{\mathbf{L}^2}(\Omega),~ \text{for } s >0.$
\end{notation} 
\section{Literature}
%\vspace{0.1cm}

\noindent The model of interest describes the interaction between a three-dimensional homogeneous viscous incompressible fluid and a two-dimensional thin elastic plate. The system is well-established in both physical and mathematical literature \cite{avalos2014mixed,chueshov2011global2}. Notably, these mathematical models and other FSI PDE dynamics are closely linked to the modeling of problems in aeroelasticity, biomedical applications, --particularly the study of eye mechanics \cite{bodnar2014fluid,hsu2011blood,maklad2018fluid}--, and the phenomenon known as ``sloshing" \cite{ibrahim2005liquid}. 

The linearized version--i.e., Stokes fluid flow rather than Navier-Stokes flow-- of the particular FSI model in \eqref{1}-\eqref{40}, was initially studied in \cite{chueshov2011global2}, where the authors established the well-posedness via the Galerkin method, and proved the existence of a compact finite-dimensional global attractor in the case that the plate equation is under the influence of the von Karman nonlinearity.  Moreover, in \cite{avalos2014mixed}, the authors invoked a strongly continuous semigroup approach for the linear FSI system under consideration and derived a variational formulation to show the semigroup well-posedness of the PDE system. At this point we should note that these methodologies are not applicable for the analysis of the current nonlinear PDE system due to the presence of the Navier-Stokes nonlinearity $(u\cdot\nabla)u$ and subsequent lack of hyper dissipativity.  

On the other hand, the body of work with respect to the quantitative properties and numerical analysis of such 3D Stokes/Navier-Stokes flow and 2D plate interaction FSI systems is not extensive \cite{barbu2007existence,burman2022stability,kuberry2013decoupling,kukavica2010strong,geredeli2024partitioning}. Numerical techniques for FSI problems with moving or non-moving interfaces include monolithic or partitioned methods \cite{barbu2007existence,barbu2008smoothness, burman2022stability, kuberry2013decoupling, kukavica2010strong}. With respect to these methods, a monolithic approach is not desirable for the system \eqref{1}-\eqref{40}, since the fluid and structure problems require different levels of solution regularity at the finite energy level, i.e., $H^1$-fluid test functions versus $H^2$-plate test functions. Moreover, the presence of the biharmonic operator in the plate dynamics presents a numerical challenge that arises from the necessity of utilizing finite elements of class $C^1$ for conforming FEM approximation, while a conforming FEM for fluid velocity can be obtained via $C^0$ class elements. However, in a recent work \cite{geredeli2024partitioning}, the authors 
considered the ``linearized" version (with Stokes flow instead of Navier-Stokes) of the PDE model in \eqref{1}-\eqref{40}, and used a domain decomposition approach to solve respective fluid and plate subproblems. Subsequently, the authors in \cite{geredeli2024partitioning} generated a time discretization scheme to approximate the solutions to a time dependent Stokes-plate interaction PDE problem, and established the stability of their numerical scheme.

\section{Novelty and Challenges Encountered}

\noindent With the objective of understanding the qualitative and quantitative properties of the nonlinear FSI system in \eqref{1}-\eqref{40}, we develop an efficient algorithm to numerically approximate the corresponding solutions. This requires one to establish the well-posedness of the coupled static PDE system \eqref{var}, after the time derivatives of solution variables are written (informally) as finite difference quotients. To this end, we will appeal to the abstract theory introduced in \cite{girault2012finite}, specially crafted for nonlinear mixed variational problems, wherein an inf-sup approach is followed to show existence-uniqueness of solutions under a small data assumption. The primary challenge in our analysis is to have precise control of the Navier-Stokes nonlinearity $(u\cdot \nabla)u$ in the fluid PDE component with a view of deriving a mixed variational formulation to solve the solution variables of the entire PDE system \eqref{var}. In this regard, the main challenges associated with the analysis and the novelties are as follows:\\ 

1) \underline{Lack of a direct elimination of pressure $p(x,t)$:} The coupling of the parabolic and hyperbolic dynamics is given via both the matching velocity terms $u_3|_{\Omega_p} = w_t,$ and the flux of the fluid (which is essentially the Dirichlet trace of the fluid pressure $p|_{\Omega_p})$, as opposed to the matching fluxes of the fluid and plate dynamics. This prohibits a meaningful application of the Leray projector-- which is standard for uncoupled Navier-Stokes equations with homogeneous boundary condition-- to eliminate the pressure. Moreover, the Navier-Stokes nonlinearity in the fluid PDE component does not allow for identification of the pressure as the solution of inhomogeneous BVP (in the style of  \cite{avalos2014mixed}), with the boundary data depending on fluid velocity and plate variables. Therefore, we alternatively consider the pressure $p(x, t)$ as a solution        variable to be directly solved within an intrinsic weak variational formulation \eqref{var}. In this connection, the pressure variable appears as a constraint in our formulation and its Dirichlet trace  $p|_{\Omega_p}$ is eventually recovered through this mixed formulation.\\

2) \underline{Formulation of nonstandard mixed variational problem:} As we said, the analysis the quantitative properties and numerical approximations of 3D Stokes/Navier-Stokes flow and 2D plate interaction FSI systems is not extensive. Moreover, the body of work with respect to the well-posedness of the static FSI model in (\ref{var}) was shown for strictly linearized versions--i.e., Stokes fluid flow rather than Navier-Stokes flow--or for nonlinear versions of \eqref{1}-\eqref{40} in which the nonlinearity appears \textit{only} strictly in the plate component \cite{chueshov2011global2, avalos2014mixed}. While \cite{chueshov2011global2} employed a Galerkin approximation for the Stokes-nonlinear plate interaction FSI model, \cite{avalos2014mixed} invoked a strongly continuous semigroup approach for the fully linear FSI system. However, these methodologies are not applicable for the analysis of nonlinear FSI due to the presence of the Navier-Stokes nonlinearity $(u\cdot\nabla)u.$ In this regard, one of our objectives is to derive a mixed variational approach to obtain the well-posedness of the static coupled PDE system \eqref{var}. This mixed formulation will lead to the derivation of a numerical approximation scheme. For this, we reformulate the static ($\lambda$-parametrized) problem in a ``nonstandard" variational form and appeal to the abstract framework introduced for \underline{nonlinear} PDE models  \cite{girault2012finite}. Our major challenges here are to (i) show the ellipticity and weak continuity condition for the ``nonlinear mapping" $A(\cdot; \cdot,\cdot)$ (the associated ``trilinear form"), and (ii) prove the inf-sup condition for the bilinear form $B(\cdot,\cdot)$ which models the constraint. Moreover, for the uniqueness of the solution, we derive some necessary size conditions on the problem data. \\

3) \underline{Semi-discrete FEM Approximation:} We will invoke the mixed variational formulation derived in \cite{avalos2014mixed} above in the derivation of a companion mixed FEM approximation scheme to approximate the solutions to \eqref{1}-\eqref{40}. This is done through a combination of $\mathbb{P}2/\mathbb{P}1$ Taylor-Hood elements and appropriately $H^2$ conforming structure elements. In this connection, our major challenge is the presence of two non-linear terms, the Navier-Stokes nonlinearity, $(u\cdot\nabla)u$, and its associated flux term on the plate component, $\frac{1}{2}u_3^2$. To address these issues, we appeal to a Picard iteration scheme (ostensibly an Oseen approximation of Navier-Stokes) which allows us to approximate the advective terms with information gained in the previous iteration. Additionally, as stated above, the pressure $p(x,t)$ is not eliminated from the plate component. This leads to a data structure translating between the structure and pressure elements, which allows us to compute the appropriate integrals using fundamental properties of finite element bases.

\section{Preliminaries}
As mentioned above, our main goal is to show the existence-uniqueness (under a small data assumption) of the weak solutions to the static version of the PDE system \eqref{1}-\eqref{40} and establish finite element method (FEM) approximations to the corresponding solution. To this end, in addition to the notation given above we also introduce the following spaces for the fluid and plate solution variables of the coupled system \eqref{1}-\eqref{40}: for the fluid velocity we define 
\begin{align*}
U = \{ v = (v_1, v_2, v_3) \in \mathbf{H}^1(\Omega_f) : v_1 = v_2 = 0 \text{ on } \Omega_p, \, v = 0 \text{ on } S \}.
\end{align*}  
For the appropriate treatment of the plate solution variables, we use the fact that $\widehat{\mathbf{H}^2_0}(\Omega_p)\subset \mathbf{H}_0^2(\Omega_p)$, and we critically utilize the following decomposition described in \cite{chueshov2011global1}:
\begin{align}\mathbf{H}_0^2(\Omega_p)=\widehat{P}\mathbf{H}_0^2(\Omega_p)\oplus (I-\widehat{P})\mathbf{H}_0^2(\Omega_p), \label{4.5g}\end{align} where $\widehat{P}:\mathbf{H}_0^2(\Omega_p)\rightarrow \widehat{\mathbf{H}^2_0}(\Omega_p)$ is the orthogonal projection with respect to the inner product $(\Delta(\cdot),\Delta(\cdot))_{\Omega_p},$ and and so a fortiori, $$(I-\widehat{P})\mathbf{H}_0^2(\Omega_p)=span\{\xi \in \mathbf{H}_0^2(\Omega_p): \Delta^2 \xi=1; ~~\xi = \frac{\partial \xi}{\partial n_p} = 0\}.$$ Subsequently, with respect to the above notation, we introduce our state space
\begin{align*}
\Sigma = \{ (\phi, \psi) \in U \times \widehat{\mathbf{H}^2_0}(\Omega_p) : \phi_3\rvert_{\Omega_p} = \psi \},
\end{align*}
with the associated norm
\begin{align*}
\lVert (u, v) \rVert_{\Sigma} = \lVert u \rVert_{U} + \lVert v \rVert_{\widehat{\mathbf{H}^2_0}(\Omega_p)}. 
\end{align*}  
With the objective of understanding the quantitative properties of the coupled PDE system \eqref{1}-\eqref{40} and generating a time-dependent finite element method (FEM) scheme, we appeal to the abstract difference approximation method (see \cite{showalter1996monotone}, page 127). The underlying idea is to approximate the time component of PDE system \eqref{1}-\eqref{40} by finite-difference equations. That is, the following finite difference quotients are invoked to approximate the time derivatives of the fluid and plate variables for $\epsilon>0$ sufficiently small:
%and dual space of $\Sigma$ is defined as
%\[
%\lVert (u, v) \rVert_{\Sigma^{'}} = \sup_{(\phi, \psi) \in \Sigma} \frac{(( u,v), (\phi, \psi))}{\lVert (\phi, \psi) \rVert_{\Sigma}}.
%\]
\begin{align}
\label{0.14c}
u_t(t) = \frac{u(t) - u(t - \epsilon)}{\epsilon} = \frac{1}{\epsilon}\big[ u(t) - u(t - \epsilon)\big] = \lambda u - u^{\ast}.
\end{align}
Here, $\lambda =\frac{1}{\epsilon},$ and $u^{\ast} = \frac{1}{\epsilon} u(t - \epsilon)\in U$ is assumed to be known since it involves ``preceding times". For the plate components, by using a change of variable and similarly applying a finite difference quotient we get
\begin{align}
\label{0.14cc}
&w_1:=w,\nonumber\\
&w_2:=w_t=(w_1)_t \to \lambda w_1-w_1^*,\nonumber\\
&w_{tt} = (w_t)_t = (w_2)_t \to \lambda w_2 - w_2^{\ast}, 
\end{align}
where $\{w_1^*,w_2^*\}\in \mathbf{H_0^2}(\Omega_p)\times \mathbf{H^1}(\Omega_p).$ Subsequently, approximating the original PDE model \eqref{1}-\eqref{40} via \eqref{0.14c}-\eqref{0.14cc}, one can derive the following $\lambda$-parametrized static PDE system:
\begin{align}
\begin{cases}
&\lambda u -\nu \Delta u + (u \cdot \nabla) u + \nabla p_0 = u^{\ast}~~~~~~~~~~~ \text{in} \ \Omega_f, \\
&\nabla \cdot u = 0~~~~~~~~ \text{in} \ \Omega_f, \\
&w_1 = \frac{1}{\lambda} w_2 + \frac{1}{\lambda} w_1^*~~~~~~~~~~ \text{in} \ \Omega_p, \\
&\lambda w_2 - \rho \Delta (\lambda w_2) + \frac{1}{\lambda}\Delta^2(w_2) = p|_{\Omega_p} + \frac{1}{2} u_3^2|_{\Omega_p}+w_2^{\ast}- \rho \Delta (w_2^{\ast})-\frac{1}{\lambda}\Delta^2(w^*_1)~~~~~ \text{in} \ \Omega_p, \\
&w_1 = \frac{\partial w_1}{\partial n_p} = 0\quad \text{on}~~~~ \partial \Omega_p, \\
&u = 0~~~~ \text{on } S, \\
&[u_1,u_2,u_3] = [0,0,w_2] ~~~ \text{on } \Omega_p,
\end{cases}
\label{2.4}
\end{align}
where the associated fluid pressure $p=p_0+c_0$ with $p_0 \in \widehat{L}^2(\Omega_f)$ and $c_0 \in \mathbb{R}$. 

\section{Main Result I: well-posedness of the \underline{static} Navier-Stokes fluid-plate interaction PDE System \eqref{1}-\eqref{40}}
\noindent This section is devoted to establishing the existence and uniqueness (under small data assumption) of the solution to the $\lambda-$parametrized static PDE system \eqref{2.4}. For this, we appeal to a methodology to solve general nonlinear variational formulations which were initially introduced in \cite{girault2012finite} for static ``uncoupled Navier-Stokes equations" wherein an inf-sup approach is followed to show existence-uniqueness of solutions under a small data assumption. Our main result is given as follows:
\begin{theorem}
\label{Th0.1}
With reference to the system \eqref{2.4}, let $ u^* \in U $, $ \{w_1^*,w_2^*\} \in \mathbf{H_0^2}(\Omega_p) \times \mathbf{H^1}(\Omega_p) $. Then, there exists at least one  solution $ \{(u, w_2),w_1, p\} \in \Sigma \times \widehat{\mathbf{H}^2_0}(\Omega_p)\times \mathbf{L}^2(\Omega_f)$ of \eqref{2.4}. Moreover, this solution is unique if  
\begin{align} ( \lVert u^*\rVert_{\mathbf{L}^2(\Omega_f)}^2 +  \lVert w_2^*\rVert_{\mathbf{L}^2(\Omega_p)}^2 + \rho \lVert \nabla w_2^*\rVert_{\mathbf{L}^2(\Omega_p)}^2  + \lVert \Delta w_1^*\rVert_{\mathbf{L}^2(\Omega_p)}^2) \leq \widetilde C \min\{\lambda^{\frac32}\nu^{\frac52}, \rho \lambda^4\},
\end{align}
where $\nu>0$, $\lambda>0$, $\rho>0$, and $\widetilde{C}>0$ is some constant independent of $\nu,\lambda,\rho$. 
\end{theorem}

\begin{proof}
The proof of Theorem \ref{Th0.1} is based on the abstract framework introduced in \cite{girault2012finite}. For the reader's convenience, we give this framework in detail in Appendix. Our modus operandi is to generate and solve a variational formulation in solution variables $\{u,w_2\}$ and the pressure component $p_0$ of $p=p_0+c_0$, and then recover the solution variable $w_1$ via the relation \eqref{2.4}$_3$ and the constant component $c_0$ via the plate equation \eqref{2.4}$_4.$ \\

\noindent We begin with multiplying the equations \eqref{2.4}$_1$ and \eqref{2.4}$_4$ by the test functions $(\phi,\psi)\in \Sigma,$ and integrating by parts. This gives
\begin{align}
\label{3.5z}
\begin{cases}
    \lambda ( u, \phi )_{\Omega_f} + \nu ( \nabla u, \nabla \phi )_{\Omega_f} 
    + ( (u \cdot \nabla) u, \phi)_{\Omega_f} + ( p_0, \phi)_{\Omega_p} 
    - ( p_0, \text{div} \, \phi )_{\Omega_f}  
    =  ( u^{\ast}, \phi )_{\Omega_f},\\ \vspace{0.3cm}   
    \lambda ( w_2, \psi )_{\Omega_p} 
    + \lambda ( \rho \nabla w_2, \nabla \psi )_{\Omega_p} 
    + \frac{1}{\lambda} ( \Delta w_2, \Delta \psi )_{\Omega_p} =  ( p_0, \psi )_{\Omega_p}+( c_0, \psi )_{\Omega_p} + \frac{1}{2} ( u_3^2, \psi )_{\Omega_p} \\
    \hspace{80mm}+ ( w_2^{\ast}, \psi )_{\Omega_p} + \rho ( \nabla w_2^{\ast}, \nabla \psi )_{\Omega_p}  - \frac{1}{\lambda} ( \Delta w_1^*, \Delta \psi )_{\Omega_p}.
\end{cases}
\end{align}
Adding the equations in \eqref{3.5z} and taking into account that $(\phi, \psi) \in \Sigma $ (hence $( c_0, \psi )_{\Omega_p}=0$), we obtain the following variational formulation
\begin{align}
\label{var}
\begin{cases}
 A((u, w_2); (u, w_2), (\phi, \psi)) + B((\phi, \psi), p_0) = F((\phi, \psi)),~~ \forall (\phi,\psi)\in \Sigma, \\
 B((u, w_2), q) = 0,~~ \forall q\in \mathbf{L}^2({\Omega_f}).
\end{cases}
\end{align}
Here, $ A(\cdot; \cdot, \cdot) : \Sigma \times \Sigma \times \Sigma \to \mathbb{R}$ is defined as
\begin{align}
\label{A}
A((v, z); (\bar{v}, \bar{z}), (\eta, \chi)) = A_0((\bar{v}, \bar{z}), (\eta, \chi)) + A_1((v, z); (\bar{v}, \bar{z}), (\eta, \chi)),
\end{align}  
where $A_0:\Sigma \times \Sigma \to \mathbb{R}$ is the bilinear form
\begin{align}
A_0((\bar{v}, \bar{z}), (\eta, \chi)) = \lambda ( \bar{v}, \eta )_{\Omega_f} + \nu ( \nabla \bar{v}, \nabla \eta )_{\Omega_f} + \lambda ( \bar{z}, \chi )_{\Omega_p} + \lambda \rho ( \nabla \bar{z}, \nabla \chi)_{\Omega_p} + \frac{1}{\lambda} ( \Delta \bar{z}, \Delta \chi )_{\Omega_p},
\label{11.2F}
\end{align} 
and $ A_1: \Sigma \times \Sigma \times \Sigma \to \mathbb{R}$ is defined as a trilinear form given by  
\begin{align}
A_1((v, z); (\bar{v}, \bar{z}), (\eta, \chi)) = ( (v \cdot \nabla) \bar{v}, \eta)_{\Omega_f} - \frac{1}{2}  ((v \cdot n) \bar{v}_3, \chi)_{\Omega_p}.
\label{11.3}
\end{align}
Note that by the Sobolev embedding Theorem and Trace Theorem - i.e., $v\rvert_{\Omega_p},\bar{v}\rvert_{\Omega_p} \in H^{\frac{1}{2}}(\Omega_p)\subset L^4(\Omega_p)$ - the trilinear form $A_1(.;.,.)$ is well-defined. In addition, the bilinear form $ B : \Sigma \times \widehat{\mathbf{L}}^2(\Omega_f) \to \mathbb{R}$ is defined as  
\begin{align*}
B((\eta, \chi), l) = -( l, \operatorname{div} \, \eta )_{\Omega_f},
\end{align*} 
and the function $ F : \Sigma \to \mathbb{R}$ is given by
\begin{align*}
F(\eta, \chi) =  ( u^{\ast}, \eta )_{\Omega_f} + ( w_2^{\ast}, \chi )_{\Omega_p} + \rho ( \nabla w_2^{\ast}, \nabla \chi )_{\Omega_p} - \frac{1}{\lambda} ( \Delta w_1^{\ast}, \Delta \chi )_{\Omega_p}.
\end{align*}  
Now, in order to appeal Theorem \ref{ex} in Appendix, we need to show that assumptions \textbf{(A.1)-(A.3)} are satisfied for our variational problem \eqref{var}. For this, we continue with the following steps:\\

\noindent \textbf{\underline{Step I:}} To verify assumption \textbf{(A.1)}, we first recall the map $A(\cdot;\cdot,\cdot)$ defined in \eqref{A}. For any given $[u, w_2] \in \Sigma$, we then obtain
\begin{align}
A((u, w_2); (u, w_2), (u, w_2))=
&\lambda (u, u)_{\Omega_f} + \nu (\nabla u, \nabla u)_{\Omega_f} + \lambda (w_2, w_2)_{\Omega_p} + \lambda \rho (\nabla w_2, \nabla w_2)_{\Omega_p}
\nonumber\\
&
+ \frac{1}{\lambda} (\Delta w_2, \Delta w_2)_{\Omega_p}
+ ((u \cdot \nabla) u , u)_{\Omega_f} - \frac{1}{2} ((u \cdot \nu) u_3, w_2 )_{\Omega_p}
\label{12ba}
\end{align}
 Integrating by parts in \eqref{12ba} with respect to the $\Omega_f$ trilinear term and using the boundary condition $u_3\rvert_{\Omega_p} = w_2$, we obtain 
\begin{align}
A((u, w_2); (u, w_2), (u, w_2))&= \lambda \int_{\Omega_f} \lvert u\rvert^2 \, dx + \nu \int_{\Omega_f} \lvert \nabla u\rvert^2 \, dx + \lambda \int_{\Omega_p} \lvert w_2\rvert^2 \, dx 
\nonumber\\
&
+ \lambda \rho \int_{\Omega_p} \lvert \nabla w_2\rvert^2 \, dx + \frac{1}{\lambda} \int_{\Omega_p} \lvert \Delta w_2\rvert^2 \, dx
\nonumber\\
& \geq C \big(\lVert u \rVert_{\mathbf{H^1}(\Omega_f)}^2 + \lVert w_2 \rVert_{\mathbf{H^2}(\Omega_p)}^2\big)
\nonumber\\
&\geq C \lVert (u,w_2)\rVert_{\Sigma}^{2},
\end{align}
where $C=C(\lambda, \rho)>0$ is a constant. Therefore, \textbf{(A.1)} holds for problem \eqref{var}.\\

\noindent \textbf{\underline{Step II:}} To continue with the verification of \textbf{(A.2)}: Let $\{ (u_m, w_{2m}) \} \subset \Sigma$ and $(u, w_2) \in \Sigma$
satisfy 
\begin{align}(u_m, w_{2m}) \rightharpoonup (u, w_2)\; (\textrm{weakly})~ in~  \Sigma.
\label{4.3r}
\end{align}
By \eqref{4.3r} and the definition of $A_0$ in \eqref{11.2F}, it is clear that
\begin{align}
\lim_{m \to \infty} A_0((u_m,w_{2m}), (\phi, \psi)) = A_0((u,w_2), (\phi, \psi)).\label{a}
\end{align}
To see the weak continuity of the trilinear form $\{A_1((u_m,w_{2m}); (u_m,w_{2m}),(\phi,\psi))\}$ in \eqref{11.3}, we integrate by parts and use the fact $\phi_3\rvert_{\Omega_p}=\psi$, to obtain 
\begin{align}
A_1((u_m, w_{2m}); (u_m, w_{2m}), (\phi, \psi)) 
&= - \sum_{i,j=1}^3 \int_{\Omega_f} u_{mi} u_{mj} \frac{\partial \phi_j}{\partial x_i} \, d\Omega_f + \sum_{i,j=1}^3 \int_{\partial \Omega_f} u_{mj} u_{mi} \phi_j n_i \, d (\partial \Omega_f) 
\nonumber\\
&
- \big( \frac{1}{2} (u_{m})_3^2, \psi \big)_{\Omega_p}
\nonumber\\
&= - \sum_{i,j=1}^3 \int_{\Omega_f} u_{mi} u_{mj} \frac{\partial \phi_j}{\partial x_i} \, d\Omega_f + ((u_{m})_3^2, \psi)_{\Omega_p} - \big( \frac{1}{2} (u_{m})_3^2, \psi \big)_{\Omega_p}
\nonumber\\
&= - \sum_{i,j=1}^3 \int_{\Omega_f} u_{mi} u_{mj} \frac{\partial \phi_j}{\partial x_i} \, d\Omega_f + (w_{2m}^2, \psi)_{\Omega_p} - \big( \frac{1}{2} w_{2m}^2, \psi \big)_{\Omega_p},
\label{17c}
\end{align}
for $(\phi,\psi) \in \Sigma.$ Now, for the first term on the right hand side of \eqref{17c}, consider the difference $\{u_{mi} u_{mj} \frac{\partial \phi_j}{\partial x_i} - u_i u_j \frac{\partial \phi_j}{\partial x_i} \}$. Application of Hölder's inequality gives
\begin{equation}
\lVert u_{mi} u_{mj} \frac{\partial \phi_j}{\partial x_i} - u_i u_j \frac{\partial \phi_j}{\partial x_i} \rVert_{\mathbf{L}^1(\Omega_f)} \leq \lVert \frac{\partial \phi_j}{\partial x_i} \rVert_{\mathbf{L}^2(\Omega_f)} \lVert u_{mi} u_{mj} - u_i u_j \rVert_{\mathbf{L}^2(\Omega_f)}. \label{M}
\end{equation}
 If we apply the triangle inequality to the term $\lVert u_{mi} u_{mj} - u_i u_j \rVert_{\mathbf{L}^2(\Omega_f)}$ we get
\begin{align}
\lVert u_{mi} u_{mj} - u_i u_j \rVert_{\mathbf{L}^2(\Omega_f)} \leq \lVert u_{mi} (u_{mj} - u_j) \rVert_{\mathbf{L}^2(\Omega_f)} + \lVert (u_{mi} - u_i) u_j \rVert_{\mathbf{L}^2(\Omega_f)} \label{E}.
\end{align}
For the first term of the RHS of \eqref{E}, application of Hölder's inequality yields
\begin{align*}
\lVert u_{mi} (u_{mj} - u_j) \rVert_{\mathbf{L}^2(\Omega_f)} \leq \lVert u_{mi} \rVert_{\mathbf{L}^4(\Omega_f)} \lVert u_{mj} - u_j \rVert_{\mathbf{L}^4(\Omega_f)}.
\end{align*}
We note that $\mathbf{H^1}(\Omega_f)$ compactly embedded into $\mathbf{L}^4(\Omega_f)$ by the Rellich-Kondrachov Theorem, and hence $ u_{mj} \to u_j$ strongly in $\mathbf{L}^4(\Omega_f)$. This gives that \begin{align} \lVert u_{mi} \rVert_{\mathbf{L}^4(\Omega_f)} \lVert u_{mj} - u_j \rVert_{\mathbf{L}^4(\Omega_f)} \to 0~~ as~~ m \to \infty.
\label{E-1}
\end{align}
The second term follows similarly and we obtain that
 \begin{equation}\lVert (u_{mi} - u_i) u_j \rVert_{\mathbf{L}^2(\Omega_f)} \to 
 0~~ as~~m \to \infty. \label{E-2}\end{equation} Thus, the combination of \eqref{E-1} and \eqref{E-2} gives
 \begin{align}
\lVert u_{mi} u_{mj} - u_i u_j \rVert_{\mathbf{L}^2(\Omega_f)} \to 0 \quad \text{as} \quad m \to \infty. \label{E-3}
\end{align}
Finally, using \eqref{E-3} in \eqref{M} yields
\begin{align}
\lVert u_{mi} u_{mj} \frac{\partial \phi_j}{\partial x_i} - u_i u_j \frac{\partial \phi_j}{\partial x_i} \rVert_{\mathbf{L}^1(\Omega_f)} \to 0 \quad \text{as} \quad m \to \infty.
\label{23c}
\end{align}
Moreover, for the quadratic term on the right-hand side of \eqref{17c}, we have for any $\psi \in \textbf{L}^2(\Omega_p)$,
$$
(w_{2m}^2 - w_2^2, \psi)_{\Omega_p}
= \int_{\Omega_p} (w_{2m} - w_2)(w_{2m} + w_2)\psi \, d\Omega_p,
$$
and by Hölder’s inequality 
\begin{align*}
\lvert (w_{2m}^2 - w_2^2, \psi)_{\Omega_p} \rvert
\leq \lVert w_{2m} - w_2\rVert_{\mathbf{L}^4(\Omega_p)} \, \lVert w_{2m} + w_2\rVert_{\mathbf{L}^4(\Omega_p)} \, \lVert \psi\rVert_{\mathbf{L}^2(\Omega_p)}.
\end{align*}
Since the embedding $\mathbf{H}_0^2(\Omega_p) \hookrightarrow \mathbf{L}^4(\Omega_p)$ is compact by the Rellich-Kondrachov Theorem, we have then
\begin{align}
(w_{2m}^2, \psi)_{\Omega_p} \to (w_2^2, \psi)_{\Omega_p}, \quad \text{for all}\; \psi \in \mathbf{L}^2(\Omega_p).
\label{24c}
\end{align} Thus, applying the convergences in \eqref{23c}  and \eqref{24c} to the right hand side of \eqref{17c} implies
\begin{align}
\lim_{m \to \infty} A_1((u_m,w_{2m}); (u_m,w_{2m}),(\phi,\psi))  = A_1((u,w_2); (u,w_2),(\phi,\psi))
\label{}
\end{align}
With this convergence  together with \eqref{a}, the assumption \textbf{(A.2)} holds for problem \eqref{var}.\\

\noindent \textbf{\underline{Step III:}} In order to show that the inf-sup condition in \textbf{(A.3)} holds, we consider the following boundary value problem: 
\begin{align}
\label{bb}
&\text{div} \, \phi = -q\quad \text{in} ~~ \Omega_f,
\nonumber\\
&\phi= \psi \quad \text{on} \quad \Omega_p, 
\nonumber\\
&\phi= 0 \quad \text{on }\quad S,
\end{align}
for given $q\in  \widehat{\mathbf{L}^2}(\Omega_f) $ and $\psi \in \widehat{\mathbf{H}^2_0}(\Omega_p).$ It is known that there exists a unique $\phi$ solves problem \eqref{bb} for $q \in \mathbf{L}^2(\Omega_f)$ and $\psi \in \mathbf{H}_0^{\frac{1}{2} + \epsilon}(\Omega_p)$
(see \cite{temam2024navier} and Theorem 3.33 of \cite{mclean2000strongly}). Besides, the compatibility condition $$-\int_{\Omega_f} q  dx=\int_{\Omega_p} \psi  dx= 0,$$ holds and so
\begin{align}
\label{4.8H}
\lVert \phi\rVert_{\mathbf{H^1}(\Omega_f)} \leq C (\lVert q\rVert_{\mathbf{L}^2(\Omega_f)} + \lVert \psi\rVert_{\mathbf{H}_0^2(\Omega_p)}).
\end{align}
Now, using \eqref{bb}, we get
\begin{align}
\frac{-( q, \text{div} \, \phi )_{\Omega_f}}{\lVert \phi\rVert_U + \lVert \psi\rVert_{\widehat{\mathbf{H}^2_0}(\Omega_p)}} = \frac{\lVert q\rVert_{\mathbf{L^2}(\Omega_f)}^2}{\lVert \phi\rVert_U + \lVert \psi\rVert_{\widehat{\mathbf{H}^2_0}(\Omega_p)}}.
\end{align}
If we choose now $\psi = \lVert q\rVert_{\textbf{L}^2(\Omega_f)} \tilde{\psi}$ in \eqref{bb}, where $\tilde{\psi} \in \widehat{\mathbf{H}^2_0}(\Omega_p)$ and $\lVert \tilde{\psi}\rVert_{\widehat{\mathbf{H}^2_0}(\Omega_p)} = 1$, we have
\begin{align}
\sup_{(\mu, \xi) \in \Sigma} \frac{-(q, div(\mu))}{\lVert \mu\rVert_{U} + \lVert \xi\rVert_{\widehat{\textbf{H}_0^2}(\Omega_p)}}
&\geq \frac{-(q, div(\phi))}{\lVert \phi\rVert_{U} + \lVert \psi\rVert_{\widehat{\textbf{H}_0^2}(\Omega_p)}}\nonumber\\
&=\frac{\lVert q\rVert_{\textbf{L}^2(\Omega_f)}^2}{\lVert \phi\rVert_{U} + \lVert q\rVert_{\textbf{L}^2(\Omega_f)}}
\nonumber\\
& \geq \beta \lVert q\rVert_{\textbf{L}^2(\Omega_f)},
\label{4.9}
\end{align}
where $\beta=\frac{1}{2C+1}$. Thus, the inf-sup assumption in \textbf{(A.3)} holds with $\beta= \frac{1}{2C+1}.$ Combining Step (I)-(III) and referring to Theorem \ref{ex} in the Appendix, we obtain that there exists a solution pair $(u, w_2) \in \Sigma=U \times \widehat{\mathbf{H}^2_0}(\Omega_p)$ and a function $p_0 \in \widehat{\textbf{L}^2}(\Omega_f)$ that solve the variational problem \eqref{var}. \\

\noindent\textbf{The completion of the proof of Theorem \ref{Th0.1}}\\

\noindent Having established the solution variable $w_2\in \widehat{\mathbf{H}^2_0}(\Omega_p),$ we recover the structure component $w_1$ via the relation \eqref{2.4}$_3$ for given data $w_1^*\in \mathbf{H}^1(\Omega_p).$ 
In addition, we infer from the constraint equation in \eqref{var}, the fact that
$$
u_3|_{\Omega_p} = w_2 \in \widehat{\mathbf{H}^2_0}(\Omega_p),
$$
and Green's identity that
\begin{align}
div (u) = 0 \quad \text{in } \Omega_f.
\label{25.50}
\end{align}
Moreover, if we take the test function $(\phi, \psi) \in \Sigma$ in \eqref{var} to be $\phi \in D(\Omega_f)$ and $\psi \equiv 0$, we then have:
$$
\lambda (u, \phi)_{\Omega_f} + \nu (\nabla u, \nabla \phi)_{\Omega_f} + (u \cdot \nabla u, \phi)_{\Omega_f} - (p_0, \operatorname{div}(\phi))_{\Omega_f}
= (u^*, \phi),
\quad \forall \phi \in D(\Omega_f).
$$
Then, we obtain from this relation and \eqref{25.50} that
solution component $u \in \textbf{H}^1(\Omega_f)$ satisfies
\begin{align}
&\lambda u - \nu \Delta u + u \cdot \nabla u + \nabla p_0 = u^* \quad \text{in } \Omega_f,
\nonumber\\
&div(u) = 0 \quad \text{in } \Omega_f
\nonumber\\
&u|_{\partial \Omega_f} =
\begin{cases}
\vec{0} & \text{on } \mathcal{S}, \\
[0, 0, w_2] & \text{on } \Omega_p
\end{cases}
\label{31}
\end{align}
Also, since the pressure function $p = p_0 + c_0$ where $c_0\in\mathbb{R},$ in order to recover the constant component $c_0$, it will be necessary to justify that $p_0 \big\rvert_{\Omega_p}$ has a well-defined meaning. To this end,  we begin with considering the fluid equation $\eqref{2.4}_1$ and apply Green's Theorem for $ \phi \in \mathbf{H^{1/2}}(\partial \Omega_f)$:
\begin{align}
-( \frac{\partial u}{\partial \nu}, \phi )_{\partial \Omega_f} + ( p_0, \phi)_{\partial \Omega_f}
&= ( -\Delta u, \phi)_{\Omega_f} + ( \nabla p_0, \phi )_{\Omega_f} - ( \nabla u, \nabla \phi )_{\Omega_f} + ( p_0, \operatorname{div} \, \phi )_{\Omega_f}
\nonumber\\
&= ( -(u \cdot \nabla) u, \phi )_{\Omega_f} - (\lambda u - u^*, \phi )_{\Omega_f} - ( \nabla u, \nabla \phi)_{\Omega_f} + ( p_0, \operatorname{div} \, \phi )_{\Omega_f}
\nonumber\\
&\leq C(\lVert u\rVert, \lVert p_0\rVert) \lVert \phi\rVert_{\mathbf{H^1}(\Omega_f)}.
\end{align}
Here we indirectly used the inequality 
\begin{align*}
\lvert (u \cdot \nabla w, z)\rvert \leq C \lVert u\rVert_{\mathbf{H^1}(\Omega_f)} \lVert w\rVert_{\mathbf{H^1}(\Omega_f)} \lVert z\rVert_{\mathbf{H^1}(\Omega_f)}
\end{align*}
(see \cite[Lemma 2.1, p. 284]{girault2012finite} for details). By the surjectivity of the Dirichlet Trace map, we thus have
\begin{align}
\label{33j}
[ \frac{\partial u}{\partial \nu} - p_0 n ]_{\partial \Omega_f} \in \textbf{H}^{-1/2}(\partial \Omega_f).
\end{align}
Since $\big[ \frac{\partial u}{\partial n}\big]_3=0$, we then have
\begin{align*}
p_0=p_0 n = \begin{bmatrix}
0 \\
0\\
p_0
\end{bmatrix} \in \textbf{H}^{-1/2}(\Omega_p).
\end{align*}
Thus, $ p_0 \big\rvert_{\Omega_p} \in \textbf{H}^{-1/2}(\Omega_p)$. 
Lastly, we determine the constant component of the associated pressure variable $p$, so that
$$
p = p_0 + c_0,
$$ 
where again $p_0 \in \widehat{\textbf{L}^2}(\Omega_f)$ is the solution of the mixed variational formulation \eqref{var}. 
What is more, the choice of $c_0$ ensures that 
aforesaid $w_2 \in \widehat{\textbf{H}_0^2}(\Omega_p)$ satisfies the plate 
equation $\eqref{2.4}_4$. 

(A) We first note that for any $\psi \in \widehat{\textbf{H}_0^2}(\Omega_p)$ 
$\exists \, \phi \in U$ such that $(\phi, \psi) \in \Sigma$. Indeed,
by Theorem 3.33, p. 95, of \cite{mclean2000strongly}, one 
can extend $\psi$ by zero to a $\textbf{H}^{\frac{1}{2}}(\partial \Omega_f)$-function. Thus, by the Sobolev Trace Theorem, 
$$
\exists \, \phi \in \textbf{H}^1(\Omega_f) \text{ s.t. } \left. \phi \right|_{\partial \Omega_f} =
\begin{cases}
\vec{0} \text{ on } \mathcal{S} \\
[0, 0, \psi] \text{ on } \Omega_p.
\end{cases}
$$
(B) For given $\psi \in \widehat{\mathbf{H}^2_0}(\Omega_p)$, let $\phi \in U$
be such that $(\phi, \psi) \in \Sigma$ (assured by (A)).
Multiplying the fluid PDE in \eqref{31} by this $\phi$,
and  integrating by parts, we have then
\begin{align}
\lambda (u, \phi) + \nu (\nabla u, \nabla \phi)_{\Omega_f} + ((u \cdot \nabla) u, \phi)_{\Omega_f} 
+ (p_0, \phi_3)_{\Omega_p} - (p_0, div(\phi))_{\Omega_f} = (u^*, \phi)_{\Omega_f}.
\label{33d}
\end{align}
Next, we consider the first equation in \eqref{var}
with said $(\phi, \psi) \in \Sigma$ (again, $\psi \in \widehat{\mathbf{H}_0^2}(\Omega_p)$ arbitrary); to wit, we have
\begin{align}
\lambda (u, \phi)_{\Omega_f}  &+ \nu (\nabla u, \nabla \phi)_{\Omega_f} + \lambda (w_2, \psi)_{\Omega_p} 
+ \lambda \rho (\nabla w_2, \nabla \psi)_{\Omega_p} + \frac{1}{\lambda}(\Delta w_2, \Delta \psi)_{\Omega_p} + ((u \cdot \nabla) u, \phi)_{\Omega_f}
\nonumber\\
&
- \frac{1}{2}(u_3^2, \psi)_{\Omega_p} - (p_0, \operatorname{div}(\phi))_{\Omega_f} = (\dot{u}^*, \phi)_{\Omega_f}
+ (w_2^*, \psi)_{\Omega_p} + \rho (\nabla w_2^*, \nabla \psi)_{\Omega_p} - \frac{1}{\lambda} (\Delta w_1^*, \Delta \psi)_{\Omega_p}
\label{34D}
\end{align}
Subtracting the relations \eqref{33d} and \eqref{34D},
we then have, $\forall \, \psi \in \widehat{\mathbf{H}_0^2}(\Omega_p)$, the variational
relation
\begin{align}
\lambda(w_2, \psi)_{\Omega_p} + \lambda \rho (\nabla w_2, \nabla \psi)_{\Omega_p} + \frac{1}{\lambda}(\Delta w_2, \Delta \psi)_{\Omega_p}
&- (p_0, \psi)_{\Omega_p} - \frac{1}{2}(u_3^2, \psi)_{\Omega_p}
\nonumber\\
&= (w_2^*, \psi)_{\Omega_p} + \rho (\nabla w_2^*, \nabla \psi)_{\Omega_p}
- \frac{1}{\lambda}(\Delta w_1^*, \Delta \psi)_{\Omega_p}. 
\label{35f}\end{align}
(Here, we are implicitly using the fact that $\phi_3\rvert_{\Omega_p} = \psi$). Let now $\xi \in \mathbf{H}_0^2(\Omega_p)$ be the solution of the
BVP
\begin{align*}
\begin{cases}
\Delta^2 \xi = 1 \quad \text{in } \Omega_p \\
\xi\rvert_{\partial \Omega_p} = \frac{\partial \xi}{\partial n_p} = 0 \quad \text{on } \partial \Omega_p
\end{cases}
\end{align*}
(Recall the orthogonal decomposition in \eqref{4.5g}). We also note that $$\int_{\Omega_p} \xi \, d\Omega_p = \int_{\Omega_p} \xi \, \Delta^2 \xi \, d\Omega_p = \| \Delta \xi \|^2_{\Omega_p} \neq 0.$$
Therewith, if
\begin{align}
c_0 &\equiv ( \int_{\Omega_p} \xi \, d\Omega_p)^{-1}
\cdot \Big[ \lambda(w_2, \xi)_{\Omega_p} + \lambda \rho (\nabla w_2, \nabla \xi)_{\Omega_p} - (p_0, \xi)_{\Omega_p}
\nonumber\\
&- \frac{1}{2}(u_3^2, \xi)_{\Omega_p} - (w_2^*, \xi)_{\Omega_p} - \rho (\nabla w_2^*, \nabla \xi)_{\Omega_p} \Big],
\label{36c}
\end{align}
then the meaningful boundary trace for $p_0\rvert_{\Omega_p}$, provided
in \eqref{33j}, allows for the well-definition of $c_0$.
With this choice of constant component $c_0$ and hence with the pressure term
\begin{align}
p = p_0 + c_0,
\end{align}
we have that $w_2 \in \widehat{\mathbf{H}_0^2}(\Omega_p)$ satisfies
\begin{align}
\lambda(w_2, \psi)_{\Omega_p} + \lambda \rho (\nabla w_2, \nabla \psi)_{\Omega_p} + \frac{1}{\lambda} (\Delta w_2, \Delta \psi)_{\Omega_p}
&= \frac{1}{2}(u_3^2, \psi)_{\Omega_p} + (p, \psi)_{\Omega_p} + (w_2^*, \psi)_{\Omega_p}
\nonumber\\
&+  \rho (\nabla w_2^*, \nabla \psi)_{\Omega_p} - \frac{1}{\lambda} (\Delta w_1^*, \Delta \psi)_{\Omega_p}
\; \forall \, \psi \in \hat{H}^2(\Omega_p),
\end{align}
after combining \eqref{35f}, \eqref{36c} and \eqref{4.5g}.\\

\noindent Thus $w_2 \in \widehat{\mathbf{H}}_0^2(\Omega_p)$ solves the plate equation 
in $\eqref{2.4}_4$. In sum, the variables $[u, w_1, w_2, p]$
obtained by the Babuška–Brezzi Theorem solve the coupled PDE system in $\eqref{2.4}_4$.\\

\noindent \underline{ \textbf{Proof of uniqueness}}: if we take the test functions $\phi=u$ and $\psi=w_2$ in \eqref{33d} and\eqref{35f}, we have 
\begin{align}\lambda (u, u)_{\Omega_f} + \nu(\nabla u, \nabla u)_{\Omega_f} + ((u \cdot \nabla) u, u)_{\Omega_f}+(p_0,u_3)_{\Omega_p} = (u^*, u)_{\Omega_f}
\label{5.3}
\end{align}
and
\begin{align} &\lambda (w_2, w_2)_{\Omega_p} + \lambda \rho (\nabla w_2, \nabla w_2)_{\Omega_p} + \frac{1}{\lambda} (\Delta w_2, \Delta w_2)_{\Omega_p}
\nonumber\\
&=(p_0,w_2)_{\Omega_p}+  \frac{1}{2} (u_3^2, w_2)_{\Omega_p} 
+ (w_2^*, w_2)_{\Omega_p}
+ \rho(\nabla w_2^*, \nabla w_2)_{\Omega_p} - \frac{1}{\lambda} (\Delta w_1^*, \Delta w_2)_{\Omega_p}.
\label{5.4}
\end{align}
By summing \eqref{5.3} and \eqref{5.4}, invoking integration by parts, and recalling that $u_3\rvert_{\Omega_p}=w_2$, one obtains cancellation of the pressure and nonlinear terms over $\Omega_f$ and $\Omega_p$. Applying H\"older’s and Young’s inequalities thereafter yields the following estimate:
\begin{align*}
\lambda \lVert u \rVert_{\mathbf{L}^2(\Omega_f)}^2&+ \nu \lVert \nabla u \rVert_{\mathbf{L}^2(\Omega_f)}^2  + \lambda  \lVert   w_2 \rVert_{\mathbf{L}^2(\Omega_p)}^2+ \lambda \rho \lVert  \nabla w_2 \rVert_{\mathbf{L}^2(\Omega_p)}^2 + \frac{1}{\lambda} \lVert  \Delta w_2 \rVert_{\mathbf{L}^2(\Omega_p)}^2
\nonumber\\
&= (u^*, u)_{\Omega_f} + (w_2^*, w_2)_{\Omega_p}+ \rho (\nabla w_2^*, \nabla w_2)_{\Omega_p} - \frac{1}{\lambda} (\Delta w_1^*, \Delta w_2)_{\Omega_p}
\nonumber\\
\leq &\frac12(\lambda\lVert  u \rVert_{\mathbf{L}^2(\Omega_f)}^2 + \lambda\lVert w_2 \rVert_{\mathbf{L}^2(\Omega_f)}^2 + \lambda \rho \lVert \nabla w_2 \rVert_{\mathbf{L}^2(\Omega_p)}^2 + \frac{1}{\lambda} \lVert \Delta w_2 \rVert_{\mathbf{L}^2(\Omega_p)}^2) 
\nonumber\\
&+ \frac{C}{\lambda}(  \lVert u^*\rVert_{\mathbf{L}^2(\Omega_f)}^2 +  \lVert w_2^*\rVert_{\mathbf{L}^2(\Omega_p)}^2 + \rho \lVert \nabla w_2^*\rVert_{\mathbf{L}^2(\Omega_p)}^2  + \lVert \Delta w_1^*\rVert_{\mathbf{L}^2(\Omega_p)}^2).
\end{align*}
Now, absorbing the corresponding terms from the right-hand side into the left-hand side gives
\begin{align}
&\lambda \lVert  u \rVert_{\mathbf{L}^2(\Omega_f)}^2+ \nu \lVert \nabla u \rVert_{\mathbf{L}^2(\Omega_f)}^2  + \lambda  \lVert  w_2 \rVert_{\mathbf{L}^2(\Omega_p)}^2+ \lambda \rho \lVert  \nabla w_2 \rVert_{\mathbf{L}^2(\Omega_p)}^2 + \frac{1}{\lambda} \lVert \Delta w_2 \rVert_{\mathbf{L}^2(\Omega_f)}^2 \nonumber\\
\leq &\frac{C}{\lambda}( \lVert u^*\rVert_{\mathbf{L}^2(\Omega_f)}^2 +  \lVert w_2^*\rVert_{\mathbf{L}^2(\Omega_p)}^2 + \rho \lVert \nabla w_2^*\rVert_{\mathbf{L}^2(\Omega_p)}^2  + \|\Delta w_1^*\lVert_{\mathbf{L}^2(\Omega_f)}^2). \label{bound}
\end{align}
In what follows we will denote the constant 
$$K:=C( \lVert u^*\rVert_{\mathbf{L}^2(\Omega_f)}^2 +  \lVert w_2^*\rVert_{\mathbf{L}^2(\Omega_p)}^2 + \rho \lVert \nabla w_2^*\rVert_{\mathbf{L}^2(\Omega_p)}^2  + \|\Delta w_1^*\lVert_{\mathbf{L}^2(\Omega_f)}^2).$$
Assume that there are two solutions $(u^1,p_0^1,w_2^1)$ and $(u^2,p_0^2,w_2^2)$ to problem \eqref{2.4} and denote $$u=u^1-u^2;~~~~~ w_2=w_{2}^{1}-w_{2}^{2};~~~~~p_0=p_0^1-p_0^2.$$
Now, re-writing the equation \eqref{33d} with respect to $(u^1, p_0^1)$ and $(u^2, p_0^2)$; and also re-writing \eqref{35f} with respect to $(w_2^1, p_0^1)$ and $(w_2^2, p_0^2)$ and then subtracting the corresponding equations and testing with functions $\phi = u$ and $\psi = w_2$ in the resulting differences, we obtain 
\begin{align}\lambda (u, u)_{\Omega_f} + \nu(\nabla u, \nabla u)_{\Omega_f} + ((u \cdot \nabla) u^1, u)_{\Omega_f}+ ((u^2 \cdot \nabla) u, u)_{\Omega_f} + (p_0,u_3)_{\Omega_p} = 0
\label{28c}
\end{align}
and
\begin{align} \lambda (w_2, w_2)_{\Omega_p} + \lambda \rho (\nabla w_2, \nabla w_2)_{\Omega_p} + \frac{1}{\lambda} (\Delta w_2, \Delta w_2)_{\Omega_p}
=(p_0,w_2)_{\Omega_p}+  \frac{1}{2} ((u_3^1+u_3^2)u_3, w_2)_{\Omega_p}.
\label{29c}
\end{align}
By summing \eqref{28c} and \eqref{29c}, invoking integration by parts, and recalling that $u_3\rvert_{\Omega_p}=w_2$ one obtains cancellation of the pressure and nonlinear terms between $((u^2\cdot \nabla)u,u)_{\Omega_f}$ and $\frac{1}{2}(u_3^2 u_3, w_2)_{\Omega_p} $. Applying Hölder’s, Poincar\'e's, Gagliardo–Nirenberg Interpolation, and Young’s inequalities, we subsequently take the following estimate:
\begin{align*}
&\lambda \lVert u \rVert_{\mathbf{L}^2(\Omega_f)}^{2} + \nu \lVert \nabla u \rVert_{\mathbf{L}^2(\Omega_f)}^{2} + \lambda \lVert w_2 \rVert_{\mathbf{L}^2(\Omega_p)}^{2} + \lambda \rho \lVert \nabla w_2 \rVert_{\mathbf{L}^2(\Omega_p)}^{2}  + \frac{1}{\lambda} \lVert \Delta w_2 \rVert_{\mathbf{L}^2(\Omega_p)}^{2} 
\\
= &- ((u \cdot \nabla) u^1, u)_{\Omega_f} +\frac12 (u_3^1 u_3, w_2)_{\Omega_p}  = - ((u \cdot \nabla) u^1, u)_{\Omega_f} +\frac12 (w_2^1 w_2, w_2)_{\Omega_p}
\\
\leq &C\|\nabla u^1\|_{\mathbf{L}^2(\Omega_f)} \|u\|_{\mathbf{L}^2(\Omega_f)}^{\frac12} \|\nabla u\|_{\mathbf{L}^2(\Omega_f)}^{\frac32} + C\|w_2^1\|_{\mathbf{L}^2(\Omega_p)} \|w_2\|_{\mathbf{L}^2(\Omega_p)} \|\nabla w_2\|_{\mathbf{L}^2(\Omega_p)}
\\
\leq & \nu\|\nabla u\|_{\mathbf{L}^2(\Omega_f)}^2 + \lambda \rho \|\nabla w_2\|_{\mathbf{L}^2(\Omega_p)}^2 + C\nu^{-3}\|\nabla u^1\|_{\mathbf{L}^2(\Omega_f)}^4 \|u\|_{\mathbf{L}^2(\Omega_f)}^2 + \frac{C}{\lambda\rho} \|w_2^1\|_{\mathbf{L}^2(\Omega_p)}^2 \|w_2\|_{\mathbf{L}^2(\Omega_p)}^2,
\end{align*}
where the constant $C>0$ is independent of the parameters $\nu, \lambda, \rho$. Simple calculations give then
\begin{align*}
    (\lambda - C\nu^{-3}\|\nabla u^1\|_{\mathbf{L}^2(\Omega_f)}^4) \|u\|_{\mathbf{L}^2(\Omega_f)}^2 + (\lambda - \frac{C}{\lambda\rho} \|w_2^1\|_{\mathbf{L}^2(\Omega_p)}^2) \|w_2\|_{\mathbf{L}^2(\Omega_p)}^2 \leq 0.
\end{align*}
If we choose
\begin{align*}
    \lambda \geq C\nu^{-3}\|\nabla u^1\|_{\mathbf{L}^2(\Omega_f)}^4, \quad \lambda^2 \geq \frac{C}{\rho} \|w_2^1\|_{\mathbf{L}^2(\Omega_p)}^2,
\end{align*}
and use the bound obtained in \eqref{bound} we take
\begin{align*}
    \lambda^3 \geq C\nu^{-5}K^2, \quad \lambda^4 \geq \frac{CK}{\rho}.
\end{align*}
Therefore, if one has 
\begin{align*}
    K\leq \widetilde C \min\{\lambda^{\frac32}\nu^{\frac52}, \rho \lambda^4\},
\end{align*}
for some constant $\widetilde C$ which only depends on the domain then we obtain that $u^1=u^2$ and $w_2^1=w_2^2$. In addition, the  uniqueness of the  pressure follows directly from \eqref{4.9} which completes the proof of Theorem \ref{Th0.1}.
\end{proof}

\section{Numerical Implementation \& Results}
\subsection{Algorithmic Setup}
Extensive numerical analysis of the time dependent linear Stokes-Plate problem was done in \cite{geredeli2024partitioning}. Additional numerical investigations involving the time dependent linear system are in \cite{AMM25} and \cite{M24}. As the present work serves as an extension of those investigations, the methods employed herein are highly similar. To facilitate self-containedness, we reiterate the important details. 

For the purposes of this work we consider the geometrical pair $(\Omega_f,\Omega_p)$ as $\Omega_f = \{(x_1,x_2):x_1,x_2\in(0,1)\}$ and $\Omega_p=\{(x_1,x_2):x_1\in(0,1),~x_2=1\}$. This geometry is convex with wedge angles at most $2\pi/3$. To handle geometries with curved boundaries, one may invoke isoparametric elements \cite{AB84}. 

On $\Omega_f$ define the triangular mesh $\mathcal{T}:=\{\mathcal{T}_\ell\}_{\ell=1}^{L}$, where each $\mathcal{T}_\ell$ is a six node $\mathbb{P}2/\mathbb{P}1$ Taylor-Hood element. Denote the basis functions associated to these elements as $\{\phi_1,\phi_2,...,\phi_M\}$ for fluid and $\{\psi_1,\psi_2,...,\psi_{M_p}\}$ for pressure.

On the plate portion of the domain, $\Omega_p$, we employ four node quintic Hermite basis functions as such elements are $H^2$-conforming. Specifically, we use the following as our standard Hermite element:
\begin{figure}[ht!]
\begin{center}
\hspace{-0.5in}
\begin{tikzpicture}
\draw[black, thick] (0,0) -- (8,0);

\filldraw[black] (0,0) circle (2pt) ;
\draw[black, thick] (0,0) circle (5pt);
\draw  node[below] at (0,-0.15) {$x_{i_1^\ell}$};
\draw node[above] at (0,0.15) {$\theta_5$};
\draw node[above] at (0,0.65) {$\theta_1$};

\filldraw[black] (8,0) circle (2pt) ;
\draw[black, thick] (8,0) circle (5pt);
\draw  node[below] at (8,-0.15) {$x_{i_2^\ell}$};
\draw node[above] at (8,0.15) {$\theta_6$};
\draw node[above] at (8,0.65) {$\theta_2$};

\filldraw[black] (8/3,0) circle (2pt) ;
\draw  node[below] at (8/3,-0.15) {$x_{i_3^\ell}$};
\draw node[above] at (8/3,0.65) {$\theta_3$};

\filldraw[black] (16/3,0) circle (2pt) ;
\draw  node[below] at (16/3,-0.15) {$x_{i_4^\ell}$};
\draw node[above] at (16/3,0.65) {$\theta_4$};

\draw node[above] at (-1,0.2) {$\partial_x w:$};
\draw node[above] at (-1,0.7) {$w:$};
\end{tikzpicture}
\end{center}
\end{figure}

Equidistant points yield good conditioning of the resultant plate mass and stiffness matrices \cite{S06} while maintaining ease of implementation. Denote the number of nodes on the plate mesh (and hence the number of Lagrange degrees of freedom) as $\tilde{M}$. From this, the degrees of freedom for the plate (Lagrange plus Hermite degrees) is given by $DOF_s=\tilde{M}+\frac{\tilde{M}+2}{3}$. With this in hand, denote the basis on these plate elements as $\{\theta_1,\theta_2,...,\theta_{DOF_s}\}$.

To handle 3d-2d geometries, one may still appeal to $\mathbb{P}2/\mathbb{P}1$ Taylor-Hood elements, replacing the six node triangles with ten node tetrahedrons. For the plate, however, one must appeal to Argyris elements if one wishes to preserve the $H^2$ conforming property of the plate basis (see p. 255 of \cite{S06}). Alternatively, one can appeal to the non-conforming Morley triangular elements \cite{M68}. While these elements are non-conforming, they have the key advantage over Argyris via their ease of implementation by having six degrees of freedom versus twenty-one. These Morley elements were successfully utilized for an FSI in \cite{geredeli2024partitioning}.

The novelty (and complication) of the present numerical scheme beyond \cite{geredeli2024partitioning} is of course the presence of the Navier-Stokes nonlinearity, $u\cdot\nabla u$, in the fluid equation and the resultant flux matching term of the right hand side of the plate equation, $\frac{1}{2}u_3^2$. To circumvent solving a nonlinear system, we employ a Picard iteration scheme as discussed in Remark 6.41 of \cite{J16}. That is, at iteration $n$, when computing iterate $n+1$, rather than computing the full nonlinear term, one approximates thusly
\begin{align*}
(u^{n+1}\cdot\nabla u^{n+1},\phi_i)_{\Omega_f}&\approx  (u^{n}\cdot\nabla u^{n+1},\phi_i)_{\Omega_f},\\
\frac{1}{2}((u_3^{n+1})^2,\theta_i)_{\Omega_p} &\approx \frac{1}{2}(u_3^nu_3^{n+1},\theta_i)_{\Omega_p}=\frac{1}{2}(w_2^nw_2^{n+1},\theta_i)_{\Omega_p}.
\end{align*}
Such an approximation is tantamount to an Oseen approximation to Navier-Stokes, and has the effect of introducing the matrices $B^n$, $T^{n}_x$, $T^{n}_y$, and $T^{n}_z$ (if doing 3d fluids). Additionally, while the right hand side of the plate equation has a term $\frac{1}{2}u_3^2$ (or $\frac{1}{2}u_2^2$ in this example), we leverage the boundary velocity coupling on $\Omega_f$, $u_3=w_2$, to simplify the calculation before applying the appropriate Picard iterative scheme. Thus, at each iteration $n$, these matrices are given by
\begin{align}
B^n&:=\frac{1}{2}\begin{bmatrix}(w_2^{n}\theta_j,\theta_i)_{\Omega_p}
\end{bmatrix}_{i,j=1}^{DOF_s},\\
T^n_x&:=\begin{bmatrix}(u_1^n\partial_x\phi_j,\phi_i)_{\Omega_f}
\end{bmatrix}_{i,j=1}^M,\\
T^n_y&:=\begin{bmatrix}(u_2^n\partial_y\phi_j,\phi_i)_{\Omega_f}
\end{bmatrix}_{i,j=1}^M,\\
T^n_z&:=\begin{bmatrix}(u_3^n\partial_z\phi_j,\phi_i)_{\Omega_f}
\end{bmatrix}_{i,j=1}^M.
\end{align}
Additionally, we require ``mean-enforcing'' vectors $E_p$ and $E_s$ which ensure the pressure and plate velocity are mean free respectively:
\begin{align}
    E_p &:= [(\psi_i,1)_{\Omega_f}]_{i=1}^{M_p},\\
    E_s &:= [(\theta_i,1)_{\Omega_p}]_{i=1}^{DOF_s}.
\end{align}
Finally, we also require the standard mass, divergence and stiffness finite element matrices coming from the constrained variational formulation:
\begin{align}
M_f&:=\begin{bmatrix}(\phi_j,\phi_i)_{\Omega_f}
\end{bmatrix}_{i,j=1}^M\label{matrix1},\\
M_s&:=\begin{bmatrix}(\theta_j,\theta_i)_{\Omega_p}
\end{bmatrix}_{i,j=1}^{DOF_s}\label{matrix2},\\
K_f&:=\begin{bmatrix}(\nabla\phi_j,\nabla\phi_i)_{\Omega_f}
\end{bmatrix}_{i,j=1}^M\label{matrix3},\\
K_s&:=\begin{bmatrix}(\theta'_j,\theta'_i)_{\Omega_p}
\end{bmatrix}_{i,j=1}^{DOF_s}\label{matrix4},\\
S&:=\begin{bmatrix}(\theta''_j,\theta''_i)_{\Omega_p}
\end{bmatrix}_{i,j=1}^{DOF_s}\label{matrix5},\\
B_x&=\begin{bmatrix}-(\partial_x\phi_j,\psi_i)_{\Omega_f}
\end{bmatrix}_{i,j=1}^{M_p,M}\label{matrix6},\\
B_y&=\begin{bmatrix}-(\partial_y\phi_j,\psi_i)_{\Omega_f}
\end{bmatrix}_{i,j=1}^{M_p,M}\label{matrix7}.
\end{align}

Interpolating with respect to the above bases, the spatially discretized system obtained from \eqref{2.4} is given by
\begin{align}
    \begin{cases}
        \lambda M_f u^{n+1}_1 + \nu K_fu^{n+1}_1 + T^n_xu_1^{n+1}+T^n_yu_1^{n+1}+B_xp^{n+1} = F_1,\\
        \lambda M_f u^{n+1}_2 + \nu K_fu^{n+1}_2 + T^n_xu_2^{n+1}+T^n_yu_2^{n+1}+B_yp^{n+1} = F_2,\\
        B_x^Tu_1^{n+1}+B_y^Tu_2^{n+1} + E_p\mu^{n+1} = 0,\\
        E_p^Tp^{n+1} = 0,\\
        \lambda M_sw_2^{n+1} + \rho\lambda K_s w_2^{n+1} + Sw_2^{n+1} = B^nw_2^{n+1} + (p^{n+1},\theta)_{\Omega_p} + E_ss^{n+1} + F_3,\\
        \lambda M_sw_1^{n+1} - M_sw_2^{n} = F_4,\\
        E_s^Tw_2^{n+1} = 0.
\end{cases}\label{discFSI}
\end{align}
A few things bear note in \eqref{discFSI}. First, in Picard iteration, the right hand side does not change from iterate to iterate. Additionally, much like enforcing $w_2\in \widehat{H_0^2}(\Omega_p)$ introduces a Lagrange Multiplier $s$, one can enforce the mean free component of pressure by introducing a similar Lagrange Multiplier $\mu\in \mathbb{R}$ in the compressibility component (see Section 12.2.10 of \cite{LB13}). This is more straightforward numerically than forcing each pressure basis function to be mean free. 

Also note, as was done in \cite{geredeli2024partitioning}, we implement the boundary velocity coupling by requiring that $u_2^{n+1}=w_2^n$. To accomplish this, rather than build basis functions which line up perfectly on the boundary $\Omega_f$, we invoke a variational crime \cite{SF73} and merely require that the two velocity functions $u_2$ and $w_2$ align at the fluid nodes which are on $\Omega_p$. This is an appropriate approximation of the boundary coupling for fine meshes (p. 186 of \cite{AB84}).

After solving for the fluid solution pair $(u^{n+1},p^{n+1})$ using the prescribed boundary coupling, one uses $p^{n+1}$ on the right hand side of the plate equation as a known force. In doing so, one must compute an inner product of the form $(p^{n+1},\theta_j)_{\Omega_f}$. To accomplish this, we define a data structure which identifies those elements of the $\mathbb{P}_1$ pressure mesh which have non-zero measure boundary on $\Omega_p$. Then, to apply the appropriate quadrature rule for computing $(p^{n+1},\theta_j)_{\Omega_f}$, one appeals to the fact that the basis functions $\{\psi_i\}_{i=1}^{M_p}$ constitute a partition of unity for the whole geometry $\Omega_f$, and so also for $\Omega_p$.

All of this is encapsulated in the following algorithm (c.f. Algorithm 1 of \cite{geredeli2024partitioning}):
\vspace{0.1cm}
\hrule
\vspace{0.1cm}
\textbf{Algorithm 1:} FEM solver for \eqref{2.4}
\vspace{0.1cm}
\hrule
\begin{center}
\begin{algorithmic}
\State 1. Build FSI mesh
\State 2. Compute matrices 
\For{$n = 1,2,...,N-1$} 
\State 3. Solve fluid subsystem \begin{align*}
&\lambda M_f u^{n+1}_1 + \nu K_fu^{n+1}_1 + T^n_xu_1^{n+1}+T^n_yu_1^{n+1}+B_xp^{n+1} = F_1,\\
&\lambda M_f u^{n+1}_2 + \nu K_fu^{n+1}_2 + T^n_xu_2^{n+1}+T^n_yu_2^{n+1}+B_yp^{n+1} = F_2,\\
&B_x^Tu_1^{n+1}+B_y^Tu_2^{n+1} + E_p\mu^{n+1} = 0,\\
&E_p^Tp^{n+1} = 0,\\
&[u_1^{n+1},u_2^{n+1}]|_{\Omega_f} = [0,w_2^n],~[u_1^{n+1},u_2^{n+1}]|_{S} = 0.
\end{align*}
\State 4. Solve plate subsystem \begin{align*}
    &\lambda M_sw_2^{n+1} + \rho\lambda K_s w_2^{n+1} + Sw_2^{n+1} = B^nw_2^{n+1} + (p^{n+1},\theta_j)_{\Omega_p} + E_ss^{n+1} + F_3,\\
    &\lambda M_sw_1^{n+1} - M_sw_2^{n+1} = F_4,\\
    &E_s^Tw_2^{n+1} = 0.
\end{align*}
\EndFor
\end{algorithmic}
\end{center}
\hrule

\FloatBarrier  
\subsection{Test Problem and Visual Results}
For our test problem we use the following functions:
\begin{align}
u_1 &= 6(y^2-y)(x-1)^3x^3,\\
u_2 &= -3(2y^3-3y^2)x^2(x-1)^2(2x-1),\\
w_2 &= 3x^2(x-1)^2(2x-1),\\
w_1 &= -w_2,\\
p & = x-y + 1.
\end{align}
Note that $p$ consists of two components, $q = x-y$, which is mean free, and $s=1$, which is orthogonal to mean free $H_0^2(\Omega_p)$ functions. 

Implementing Algorithm 1 with the above test problem with $N=20$ iterations yields the following results:
\vspace{-0.75cm}
\begin{figure}[h]
\centering
\begin{multicols}{2}
\includegraphics[width=1.15\linewidth]{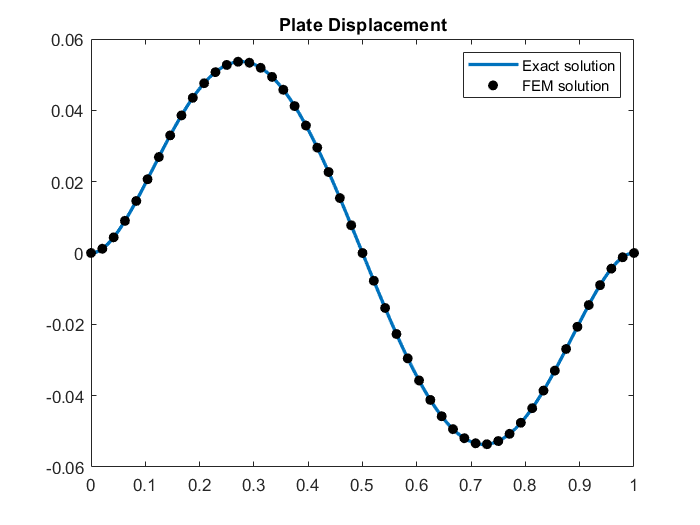}

\columnbreak

\includegraphics[width=1.15\linewidth]{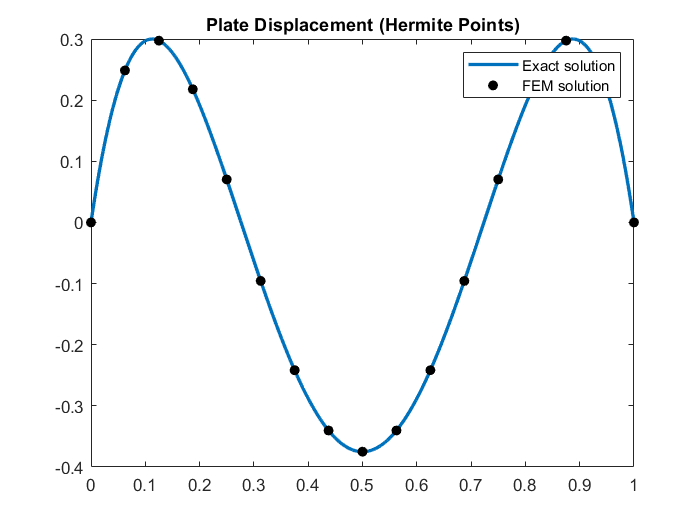}
\end{multicols}
\vspace{-0.75cm}
\caption{Plate displacement degrees of freedom ($M=2113$)}
\label{fig:plate-disp}
\end{figure}

\begin{figure}[h]
    \centering
    \begin{multicols}{2}
    \includegraphics[width=1.15\linewidth]{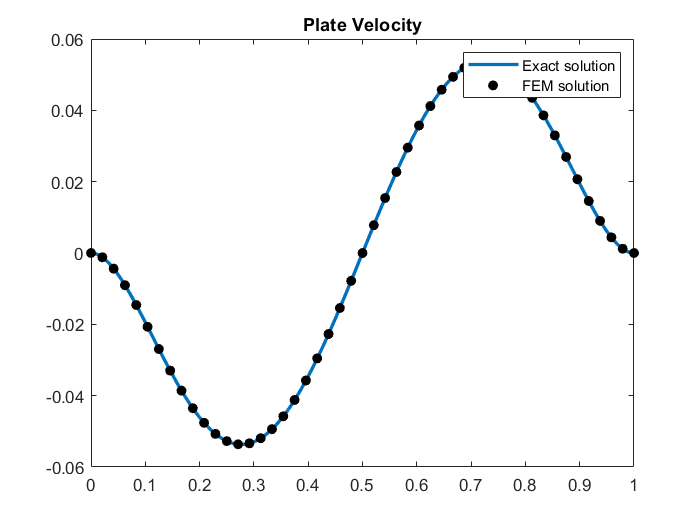}

    \columnbreak

    \includegraphics[width=1.15\linewidth]{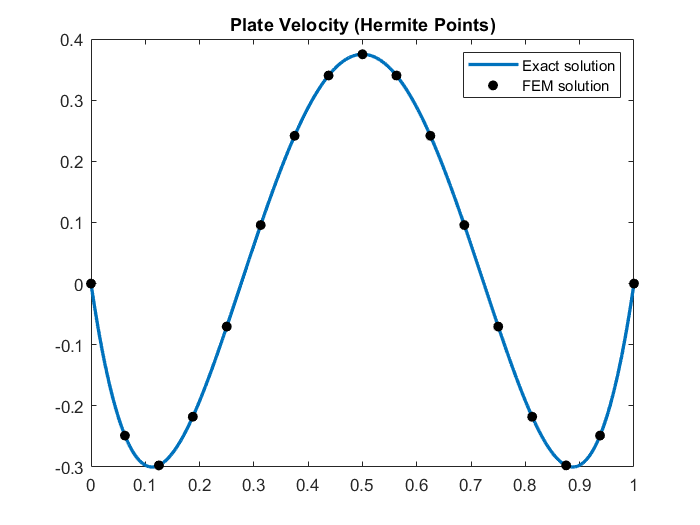}
    \end{multicols}
\vspace{-0.75cm}
    \caption{Plate velocity degrees of freedom ($M=2113$)}
    \label{fig:plate-vel}
\end{figure}
\vfill

\begin{figure}[htbp]
\centering
\begin{multicols}{2}
\includegraphics[width=1.05\linewidth]{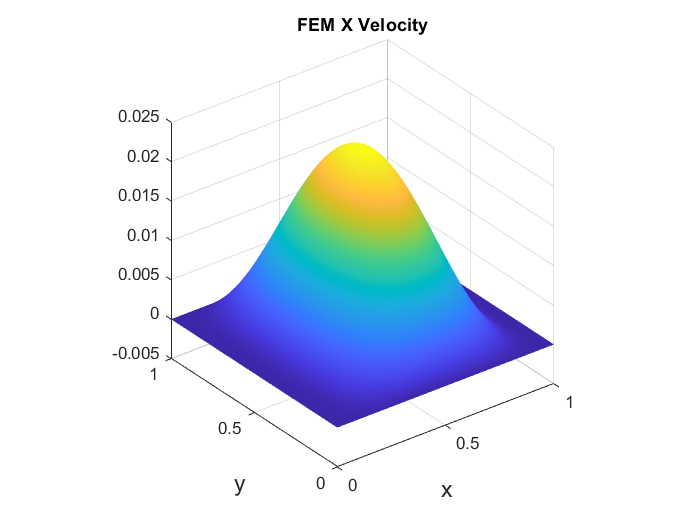}

\columnbreak

\includegraphics[width=1.05\linewidth]{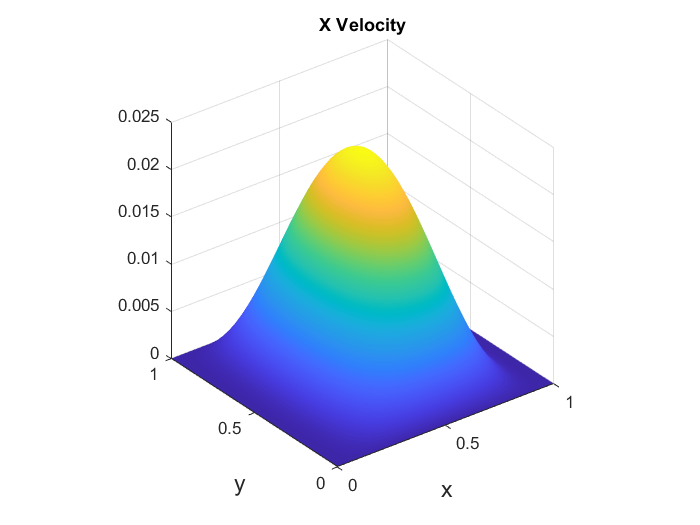}
\end{multicols}
\vspace{-0.9cm}
\caption{Finite element and actual $x$-fluid component ($M=8321$) }
\label{fig:fluid-one}
\end{figure}

\begin{figure}[htbp]
    \centering
    \begin{multicols}{2}
    \includegraphics[width=1.05\linewidth]{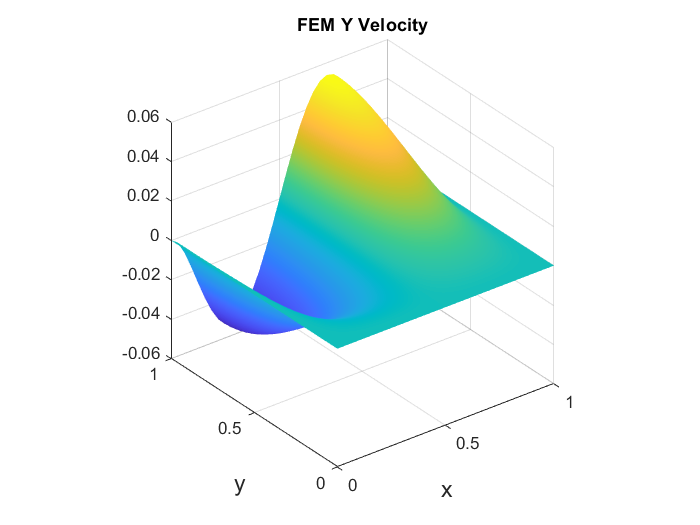}

    \columnbreak
    \includegraphics[width=1.05\linewidth]{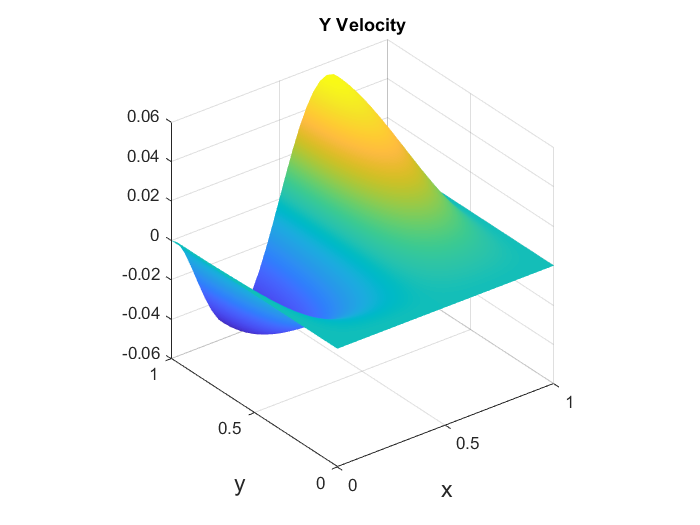}
    \end{multicols}
\vspace{-0.9cm}
    \caption{Finite element and actual $y$-fluid component ($M=8321$)}
    \label{fig:fluid-two}
\end{figure}

\begin{figure}[htbp]
    \centering
    \begin{multicols}{2}
    \includegraphics[width=1.05\linewidth]{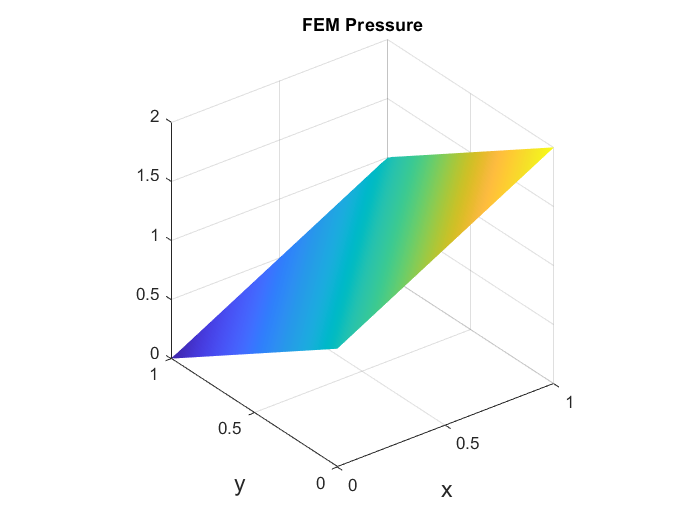}

    \columnbreak
    \includegraphics[width=1.05\linewidth]{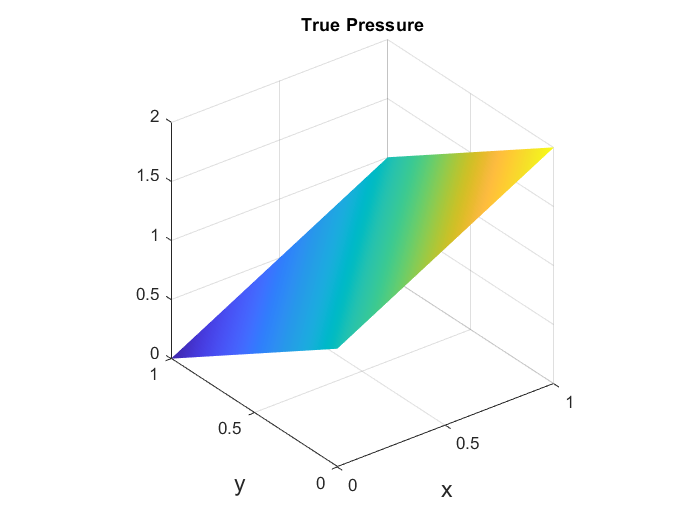}
    \end{multicols}
\vspace{-0.9cm}
    \caption{Finite element and actual pressure  ($M=8321$)}
    
    \label{fig:pressure}
\end{figure}

\FloatBarrier

\subsection{Error Analysis}
In addition to the visual results in the above section, we have the following error ratios in various norms based on characteristic length $h$:
\begin{table}[htbp]
    \begin{tabular}{c|c|c|c|c}
        $h$&$\|\nabla u_1-\nabla u_{1,I}\|_{\Omega_f}$ & $H^1$ Ratio&$\|\nabla u_2-\nabla u_{2,I}\|_{\Omega_f}$ & $H^1$ Ratio\\\hline
    0.25&	7.47E-03&	& 2.41E-02&	\\
	0.125&	1.50E-03&	2.32E+00&	6.37E-03&1.92E+00	\\
	0.0625&	3.20E-04&	2.22E+00&	1.62E-03&1.98E+00	\\
	0.03125& 7.55E-05&	2.08E+00&	4.07E-04&1.99E+00	\\
	0.015625& 1.88E-05&	2.01E+00&	1.02E-04&1.99E+00	
    \end{tabular}
\end{table}

\begin{table}[htbp]
    \begin{tabular}{c|c|c|c|c}
    $h$&$\|p-p_{I}\|_{\Omega_f}$&$L^2$ Ratio &$\|w_1''-w_{1,I}''\|_{\Omega_p}$&$H^2$ Ratio\\\hline
	0.25& 1.56E-02&	&4.09E-04&\\
	0.125& 2.59E-03&2.59E+00	&3.82E-04&9.80E-02\\
	0.0625& 3.68E-04&2.82E+00	&1.95E-04&9.69E-01\\
	0.03125& 5.08E-05&2.86E+00	&9.70E-05&1.01E+00\\
	0.015625& 1.04E-05&2.29E+00	&4.83E-05&1.01E+00\\
    \end{tabular}
\end{table}
These error ratios are as expected given the regularity of the solutions to the linear problem (see Theorem 5.6 and inequality (5.82) on page 225 of \cite{AB84}). The main difference of note from the ratios supplied in \cite{M24} is the lower order convergence on the $H^2$-norm of $w_1$ (1 above versus 2 in the cited work). This is due to the presence of rotational inertia $-\rho \Delta w_{tt}$, the introduction of which reduces the regularity of the plate displacement in the linear problem from $H^4(\Omega_p)$ to $H^3(\Omega_p)$ \cite{avalos2014mixed}.

\section{ Appendix}
\vspace{0.3cm}

\noindent \textbf{Abstract framework}\\
We consider two Hilbert spaces $V$ and $M$, equipped with the norms $\lVert \cdot \rVert_V$ and $\lVert \cdot \rVert_M$, respectively. Let $\Omega$ is a bounded domain of $\mathbb{R}^N$ with Lipschitz-continuous boundary $\partial \Omega$. We now define the function space 
\begin{align*}
V =\{ v \in H_0^1(\Omega)^N : \operatorname{div} \, v = 0\},
\end{align*}
which consists of divergence-free vector fields with $\partial \Omega = 0$. The bilinear form $ a_0(\cdot, \cdot) : V \times V \to \mathbb{R}$ is given by
\begin{align*}
a_0(u, v) = \nu \int_{\Omega} \nabla u \cdot \nabla v \, dx.
\end{align*}
The trilinear form $ a_1(\cdot; \cdot, \cdot) : V \times V \times V \to \mathbb{R}$ is defined as
\begin{align*}
a_1(w; u, v) = \int_{\Omega} (w \cdot \nabla) u \cdot v \, dx.
\end{align*}
We also define the bilinear form $b(\cdot, \cdot) : V \times M \to \mathbb{R} $ as
\begin{align*}
b(v, \mu) = -\int_{\Omega} \mu \, \operatorname{div} v \, dx.
\end{align*}
Finally, the combined form $a(\cdot; \cdot, \cdot) : V \times V \times V \to \mathbb{R}$ is given by
\begin{align*}
a(w; u, v) = a_0(u, v) + a_1(w; u, v).
\end{align*}
Based on the formulation provided above and in accordance with \cite[Chapter 4, Theorem 1.2 (p. 280) and Theorem 1.4 (p. 284)]{girault2012finite}, we present a restatement of the key components for the existence and uniqueness of weak solutions.

\begin{theorem}
\label{ex}
Consider the problem
\begin{align}
\begin{cases}
\label{3.6H}
a(u; u, v) + b(v, \lambda) = ( f, v) \quad \forall v \in \mathbf{H}_0^1(\Omega),\\
b(v, \mu)=0.
\end{cases}
\end{align}
Assume the following conditions hold:
\begin{enumerate}[label=A\arabic*.]
\item $ a(v; v, v)$ is uniformly $ V$-elliptic with respect to $w$, that is  
\begin{align*}
\exists\, \beta >0: \;a(v; v, v) \geq \beta \lVert v \rVert_V^2 \; \forall v\in V;
\end{align*}
\item \text{ The mapping } $u \to a(u, u, v) \text{ is sequentially weakly continuous on the separable space}\, V$, that is
\begin{align*}
\textrm{if}\; u_m \rightharpoonup u\; \textrm{weakly in}\;  V,\; \textrm{then}\;  
\lim_{m \to \infty} a(u_m; u_m, v) = a(u; u, v) \; \forall v\in V ;
\end{align*}
\item  The bilinear form $ b(v, \mu)$ satisfies the inf-sup condition, that is  \begin{align*}
\exists\, \beta >0: \inf_{\mu \in M} \sup_{v \in V} \frac{b(v, \mu)}{\lVert v \rVert_V \lVert \mu \rVert_M} \geq \beta > 0,
\end{align*}
\end{enumerate} 
then the problem \eqref{3.6H} has  at least one pair of solutions $(u,\lambda) \in V \times \widehat{\mathbf{L}^2}(\Omega)$.
\end{theorem}

\section{Acknowledgement}

\noindent The author Pelin G. Geredeli would like to thank the National Science Foundation, and acknowledge her partial funding from NSF Grant DMS-2348312.

\bibliographystyle{plain}
\bibliography{reference}

\end{document}